\documentclass{amsart}
\usepackage{amssymb}
\usepackage{latexsym}
\usepackage{amsthm}

\usepackage[all]{xy}

\theoremstyle{plain}

\newtheorem{thm}{Theorem}[section]
\newtheorem{cor}[thm]{Corollary}
\newtheorem{lem}[thm]{Lemma}
\newtheorem{prop}[thm]{Proposition}

\newtheorem{defn}{Definition}[section]
\newtheorem{rem}{Remark}[section]

\numberwithin{equation}{section}

\newcommand{\ZZ}{\mathbb Z}
\newcommand{\CC}{\mathbb C}

\newcommand{\PP}{\mathbb P}
\newcommand{\FF}{\mathbb F}

 \newcommand{\cO}{\mathcal O}

\begin{document}

\title{Elliptic Calabi--Yau threefolds over a del Pezzo
surface}
\author{Simon Rose}
\thanks{S. Rose was supported by a Jerry Marsden Postdoctoral Fellowship for the Fields major thematic program on Calabi-Yau Varieties: Arithmetic, Geometry and Physics from July to December 2013.}
\author{Noriko Yui}
\thanks{N. Yui was supported in part by the Natural Sciences and Engineering Research Council (NSERC) Discovery Grant.}
\address{Department of Mathematics and Statistics, Queen's University,
Kingston, Ontario Canada K7L 3N6}
\email{simon@mast.queensu.ca; yui@mast.queensu.ca}
\date{\today}

\begin{abstract}
We consider certain elliptic threefolds over the projective
plane (more generally over certain rational
surfaces) with a section in Weierstrass normal form. In particular,
over a del Pezzo surface of degree $8$, these elliptic threefolds
are Calabi--Yau threefolds. We will discuss especially
the generating functions of Gromov-Witten and Gopakumar-Vafa
invariants. 
\end{abstract}

\maketitle

\section{Introduction} \label{sect1}
\medskip
During a visit to Max-Planck-Institut f\"ur Mathematik Bonn
in the spring of 2004, Professor Hirzebruch showed the second author a
specific construction of Calabi--Yau threefolds, which
are elliptic threefolds over a del Pezzo
surface of degree $8$ in Weierstrass normal form,
that is a family of elliptic curves over a del Pezzo surface
of degree $8$ (a rational surface) \cite{hir}. 
The purpose of this short note is to discuss
the generating functions of Gromov-Witten and
Gopakumar-Vafa invariants.
\medskip

This paper was completed while both authors were in residence at the Fields Institute for the thematic program. We thank the hospitality of the Fields Institute.

\section{A del Pezzo surface of degree $8$}
\label{sect2}
\medskip

First we will give a definition of a del Pezzo surface. A good
reference is Manin \cite{Ma86}.

\begin{defn}\label{defn2.1}
{\rm A {\it del Pezzo} surface $S$ is a smooth projective geometrically
irreducible surface whose anti-canonical bundle is ample,
i.e., $-K_S$ is ample.

The {\it degree} of $S$ is a positive integer defined by
\[
\operatorname{deg} S:= K_S \cdot K_S.
\]
That is, the degree of \(S\) is the self-intersection of its canonical class.}
\end{defn}
\medskip

\begin{rem}\label{rem2.1} {\rm (1)  Every del Pezzo surface is geometrically
rational. Therefore, it is birationally equivalent to the
projective plane, $\PP^2$.

(2) Let $S$ be a del Pezzo surface. Then $1\leq  \mbox{deg}\, S\leq 9$.

(3) If $\mbox{deg}\,S >2$, then its anti-canonical bundle $-K_S$ is very ample.}
\end{rem}
\medskip

Here is a classification results of del Pezzo surfaces according to
their degrees.
\medskip

\begin{thm}\label{thm2.1}{\sl Let $S$ be a del Pezzo surface.

(a) If $\mbox\,{deg} S=4$, $S$ is birationally equivalent to
a complete intersection of two quadrics in $\PP^4$.

(b) If $\mbox\,{deg} S = 3$, $S$ is birationally equivalent to
a cubic surface in $\PP^3$.

(c) If $\mbox\,{deg} S=2$, $S$ is birationally equivalent to a hypersurface
of degree $4$ in the weighted projective $2$-space $\PP(2,1,1,1)$.

(d) If $\mbox\,{deg} S=1$, $S$ is birationally equivalent to a hypersurface
of degree $6$ in the weighted projective $2$-space $\PP(3,2,1,1)$.

(e) Any smooth surface as in (a),(b),(c) or (d) is del Pezzo surface
of the expected degree.

(f) Let $P_1, P_2,\cdots, P_r$ with $r\leq 8$ be generic points
in $\PP^2$. Let $S:=\mbox{Br}_{P_1.\cdots, P_r}(\PP^2)$
be the blow-up of $\PP^2$ at $P_i,\, 1\leq i\leq r$.
Then $S$ is a del Pezzo surface of degree $9-r$.}
\end{thm} 
\medskip

To obtain a del Pezzo surface of degree $8$, we blow-up
$\PP^2$ in one point.

\begin{cor}\label{cor2.2}{\sl
Pick a point $P\in\PP^2$
, and a line $H\subset \PP^2$ not passing through $P$. Then
\[
-K_{\PP^2}= 3H,\qquad\mbox{and}\qquad K_{\PP^2} \cdot K_{\PP^2} =9\,H^2=9.
\]
Let \(S := Bl_P(\PP^2)\) be the blow-up of \(\PP^2\) at \(P\). Furthermore, let $E$ denote the exceptional curve replacing $P$; then $E \cdot E=-1$. 
Let $\xi : S\to \PP^2$ be the blow-up map. Then
\[
K_S=\xi^*(K_{\PP^2})+E
\]
and
\begin{align*}
K_S \cdot K_S &=\xi^*(K_{\PP^2}) \cdot \xi^*(K_{\PP^2}) +2\xi^*(K_{\PP^2})\cdot E+ E\cdot E\\
&=K_{\PP^2} \cdot K_{\PP^2}+E \cdot E \\
&=9-1=8.
\end{align*}

Then $S$ is a del Pezzo surface of degree $8$.}
\end{cor}


\begin{rem} 
Let $S$ be a del Pezzo surface of degree $d$. 
Then

(1) Every irreducible curve on $S$ is exceptional.

(2) If $S$ has no exceptional curves, then either $d=9$
and $S$ is isomorphic to $\PP^2$, or $d=8$ and $S$
is isomorphic to $\PP^1\times \PP^1$.

(3) If $S$ is not isomorphic to $\PP^1\times \PP^1$, then 
the Picard group $Pic(S)$ is isomorphic to $\ZZ^{10-d}$.
In particular, if $d=8$, $Pic(S)\simeq \ZZ^2$ and is spanned by $H$
and $E$.
\end{rem}

\section{The construction of elliptic threefolds over $S$}\label{sect3}
\medskip
Let $\pi : X\to S$ be an elliptic fibration, and let $L$ be a line
bundle on $S$ with $L \cdot L =8$. Take
$$g_2\in H^0(S,L^4),\,\mbox{and}\, g_3\in H^0(S,L^6),$$
i.e., $$g_2=4L\quad\mbox{and}\quad g_3=6L$$
and let
$$X : y^2z=4x^3-g_2xz^2-g_3z^3.$$
Then the canonial bundle $K_X$ is given by
$$K_X=\pi^*(K_{X/S}+K_S)\quad\mbox{with $K_{X/S}\simeq L^{-1}$}.$$
We want $X$ to be a Calabi--Yau threefold.  The Calabi--Yau
condition imposes that
$$K_X\simeq{\mathcal O}_X\Longleftrightarrow K_S=L^{-1}
\Longleftrightarrow -K_S=L.$$
Now
$$K_S=-3H+E\Longleftrightarrow L=3H-E$$
so that
\[
4L=4(3H-E)\qquad\mbox{and}\qquad 6L=6(3H-E).
\]
\medskip

Let $[z_0:z_1:z_2]$ be the projective coordinate for $\PP^2$. Then
$g_2=g_2(z_0,z_1,z_2)\in 4L$ and is of degree $12$. While
$g_3=g_3(z_0,z_1,z_2)\in 6L$ and is of degree $18$. Put
$\Delta=4g_2^3-27g_3^2$. Then $\Delta=\Delta(z_0,z_1,z_2)\in 12L$ and
is of degree $36$.
\medskip

\section{Calculation of the Euler characteristic and the Hodge numbers}
\label{sect4}

Let $X$ be an elliptic threefold constructed above. Then
by the construction, the geometric genus of $X$ is $p_g(X)=1$
and $h^{1,0}(X)=h^{2,0}(X)=0$. So $X$ is a Calabi--Yau
threefold. Now we calculate the Euler characteristic
$e(X)$ of $X$. 
\medskip

\begin{lem}\label{lem4.1}{\sl Let $Y$ be a complex surface (possibly
with singularities). Then the Euler characteristic $e(Y)(D)$ for
any divisor $D$ is given by
$$e(Y)(D)=K_Y\cdot D-D\cdot D+\mbox{Contribution from singularities}.$$
 In particular, if $Y$ is smooth,
$$e(Y)(D)=2-2g(Y)$$
which is independent of a choice of a divisor $D$.}
\end{lem}
\medskip

\begin{prop}\label{prop4.1}{\sl Let $X: y^2z=4x^3-g_2xz^2-g_3z^3$
be a Calabi--Yau threefold over a del Pezzo surface $S$,
and let $\Delta=4g_2^3-27g_3^2$. Then
the Euler characteristic $e(X)$ of $X$ is given by the
formula
$$e(X)=e(\Delta)+\#\mbox{cusps}$$
where the Euler characteristic
$e(\Delta)$ of $\{\Delta=0\}$
is given by 
$$e(\Delta)=-\mbox{deg} K_{\Delta}=2-2g(\Delta)+2\#\mbox{cusps}$$
where $g(\Delta)$ denotes the genus of $\{\Delta=0\}$.

Moreover, we can compute that
$$\#\mbox{cusps}=192\quad\mbox{and}\quad g(\Delta)=595.$$
Finally, we obtain
$$e(X)=-480.$$}
\end{prop}

\begin{proof}
First recall that $L=-K_S$ and that $L \cdot L=8$.  Then we have 
$$K_{\Delta}=(K_S+\Delta)\cdot \Delta=(-L+12L)\cdot 12L$$
$$ =11L\cdot 12L=(11 \cdot 12)( L\cdot L)=11\cdot 12\cdot 8=1056.$$
The number of cusps is given by
$$4L\cdot 6L= 24( L \cdot L)=24\cdot 8=192.$$
Then
$$e(\Delta)=-1056+2\cdot 192=-1056+384=-672.$$

Now we need to calculate the Euler characteristic of resolutions
of singularities.  If $\{\Delta=0\}$ is smooth, its resolution is 
an elliptic curve $E$, and the Euler characteristic
$e(E)=0$. If $\{\Delta=0\}$ is a node, the Euler characteristic
of its resolution is $1$, and if $\{\Delta=0\}$ is a cusp,
the Euler characteristic of its resolution is $2$.  

Then we have
$$e(X)=\begin{pmatrix} e(E)\times e(\PP^2\setminus\Delta) \\
         +e(\Delta\setminus\{cusps\})\times
e(\mbox{resolution of a node}) \\
+\#\mbox{cusps}\times e(\mbox{resolution
of a cusp})\end{pmatrix}$$
$$
=e(\Delta)-\#\mbox{cusps}+2\#\mbox{cusps}=e(\Delta)+\#\mbox{cusps}.$$

Finally we obtain
$$e(X)=-672+192=-480.$$
\end{proof}

The Hodge nubmers $h^{1,1}(X)$ and $h^{2,1}(X)$ have
been calculated by Hulek and Kloosterman \cite{HK11} (Section 11).
This is done by calculating the Mordell--Weil rank of the
elliptic curve $\pi : X\to S$, which turns out to be $0$.

\begin{lem} {\sl 
$$h^{1,1}(X)=3,\quad\mbox{and}\quad h^{2,1}(X)=243.$$
The topological Euler characteristic is $e(X)=-480$. }
\end{lem}
Thus the Hodge diamond is given by

$$
\begin{array}{c@{}c@{}c@{}c@{}c@{}c@{}c@{}cl}
  &   &       &          1 &            &   &   & \qquad\qquad &
 B_0(X) = 1\\
  &   & 0          &            &    0 &   &   & & B_1(X) = 0\\

  & 0 &            & 3 &   & 0 &   & & B_2(X) = 3\\
1\phantom{1} &   & 243 &            & 243 &   & \phantom{1}1 & & B_3(X) =
 488\\
  & 0 &            & 3 &            & 0 &   & & B_4(X) = 3\\
  &   & 0          &            &          0 &   &   & & B_5(X) = 0\\
  &   &            &          1 &       &   &   & & B_6(X) = 1
\end{array}
$$

Recall $X$ is defined by a Weierstrass equation over the
del Pezzo surface $S$ of degree $8$ which is birational to $\PP^2$,
$$y^2z=4x^3-g_2xz^2-g_3z^3\,\quad\mbox{where } g_2, g_3\in \CC(S)$$
the $j$-invariant of $X$ is defined by
$$j=1728\displaystyle\frac{g_2^3}{\Delta}\quad\mbox{where}\quad
\Delta=4g_2^3-27g_3^2.$$

As for these elliptic threefolds, we have

\begin{lem} {\sl The $j$-invariant is a moduli for $X$.}
\end{lem}

We are also interested in the modularity question for the Galois representation associated to \(X\). However, the Betti number \(B_3(X) = 488\) is too large to make this practical. Thus, we are interested in constructing a topological mirror Calabi--Yau threefold
$\check X$.

For a topological mirror partner $\check X$ of our elliptic Calabi--Yau threefold
$X$, the Hodge numbers are
$$h^{1,1}(\check X)=243,\, h^{2,1}(\check X)=3$$
and the Euler characteristic is
$$e(\check X)=480.$$
The Betti numbers are
$$B_2(\check X)=243,\, B_3(\check X)=8.$$

In this case, the modularity of the Galois representation may at least somewhat be tractable. This leads us to ask: How can we construct such a mirror Calabi--Yau threefold?

\section{Gromov-Witten and Gopakumar-Vafa invariants}

We are naturally interested in the Gromov-Witten invariants of the threefold \(X\). These are obtained via integration against the virtual fundamental class of the moduli space of stable maps into \(X\). That is, we define
\[
N_{g,\beta}^X = \int_{[M_{g,n}(X,\beta)]^{vir}}1.
\]
In the best of cases, these invariants are positive integers and count the number of curves in \(X\) in the homology class \(\beta\). In many cases, however, since \(M_{g,n}(X,\beta)\) is a stack, the invariants are only rational numbers.

Naturally, we organize these invariants into a generating function as follows. Let \(F_g^X(q)\) and \(F^X(q,\lambda)\) be defined as
\begin{gather*}
F_g^X(q) = \sum_{\beta \in H_2(X)} N_{g,\beta}^X q^\beta \\
F^X(q) = \sum_{g=0}^\infty \lambda^{2g-2}F_g(q).
\end{gather*}

We can now define the {\em Gopakumar-Vafa/BPS} invariants via the equality
\[
F^X(q) = \sum_{g=0}^\infty \sum_{\beta \in H_2(X)} n_{g,\beta}^X \sum_{m=1}^\infty \frac{1}{k}\big(2\sin(\frac{k\lambda}{2})\big)^{2g-2}q^{k\beta}.
\]
For example, the \(g=0\) portion of this reads
\[
N_{0,\beta}^X = \sum_{\substack{\eta \in H_2(X)\\ k\eta = \beta}} \frac{1}{k^3}n_{0,\eta}^X.
\]
These invariants \(n_{g,\beta}^X\) are defined recursively in terms of the Gromov-Witten invariants \(N_{g,\beta}^X\), and {\em a priori} these are only rational numbers. It is a conjecture (see \cite{GV, Ka06}) that they are integers for all \(X, g, \beta\). We work with them because in the case of the Calabi-Yau threefold \(X\), the formul\ae\ for them turn out to be much simpler; the Gromov-Witten invariants can then be reconstructed from them.

In the case that a class \(\beta\) is primitive, the invariants \(N_{0,\beta}^X\) and \(n_{0,\beta}^X\) coincide.

\section{The geometry of \(X\)}

In order to compute these invariants, we need a bit more of a description of the geometry of the 3-fold \(X\). We begin with the following fact. The del Pezzo surface \(S\) is in fact isomorphic to the Hirzebruch surface \(\FF_1 = \PP\big(\cO \oplus \cO(1)\big)\). This is a \(\PP^1\) bundle over the base \(\PP^1\), with a \((-1)\)-curve as a section. Let \(C', F'\) denote the homology classes in \(S\) of the section and fibre, respectively. 

Consider now the following composition
\[
\xymatrix{
X \ar[r]^{p_1} \ar@/_1pc/[rr]_{\pi} & S \ar[r]^{p_2} & C.
}
\]
The generic fibre of this is an elliptically fibred K3 surface with 24 \(I_1\) fibres.

Let now \(X_F\) denote one such generic fibre, and let \(X_C\) denote the restriction of \(X\) to the section \(C\). This latter surface is a rational elliptic surface with 12 \(I_1\) fibres (which physicists call a \(\tfrac{1}{2}\)K3). Similarly, let \(C'', E''\) denote the class of the section and fibre in \(X_C\), respectively.

We want to have a description of the Picard group and a basis of \(H_2(X,\ZZ)\). So consider first the line bundles
\[
L_1 = \cO(S) \qquad L_2 = \cO(X_C) \qquad L_3 = \cO(X_F)
\]
and let \(\iota_1, \iota_2, \iota_3\) denote the respective inclusions of \(S, X_C, X_F\). We now define the homology classes
\begin{gather*}
C = (\iota_1)_*(C') = (\iota_2)_*(C'') \\
E = (\iota_2)_*(E'') \\
F = (\iota_1)_*(F').
\end{gather*}

\begin{lem}
The line bundles \(L_1, L_2, L_3\) form a basis of the Picard group of \(X\), and the classes \(C, E, F\) form a basis of \(H_2(X, \ZZ)\) (which are all effective).
\end{lem}

\begin{proof}
We can compute the intersection pairing of these bundles with these curves, which we find to be
\begin{center}
\begin{tabular}{c|c c c}
& $C$ & $F$ & $E$ \\
\hline
$L_1$ & -1& -2& 1\\
$L_2$ & -1& 1& 0\\
$L_3$ & 1& 0& 0\\
\end{tabular}
\end{center}
which clearly has determinant -1. It follows (since \(h^{1,1} = 3\)) that the lattices that these generate must be the whole lattice.
\end{proof}

We will further need the triple intersections \(\Gamma_{ijk} = \int_X L_i \smile L_j \smile L_k\), which are computed as follows.

\begin{lem}\label{lem_triple_intersections}
The triple intersections are given by the following.
\begin{align*}
\Gamma_{111} &= 8  & \Gamma_{112} &= -1& \Gamma_{113} = -2\\
\Gamma_{122} &= -1& \Gamma_{123} &= 1 & \Gamma_{133} = 0\\
\Gamma_{222} &= 0& \Gamma_{223} &= 0& \Gamma_{233} = 0\\
&&\Gamma_{333} &=0
\end{align*}
\end{lem}

\begin{proof}
These are computed simply by restricting the line bundles to the smooth representatives \(S, X_C, X_F\), where the intersections are easy to compute.
\end{proof}

\subsection{The rational elliptic surface}

The rational elliptic surface \(X_C\) is realizable as the blowup of \(\PP^2\) at 9 points in general position; as such, its intersection form is \(\Gamma_{1,9}\), where \(\Gamma_{a,b}\) is the lattice with diagonal intersection form given by 
\[
\operatorname{diag}(\underbrace{1, \ldots, 1}_a, \underbrace{-1, \ldots, -1}_b).
\]
Let \(H, C_0, \ldots, C_8\) denote the classes of the line  and exceptional curves, respectively. Then \(3H - \sum_{i=0}^8 C_i\) and \(C_0\) span a sublattice which is isomorphic to \(\Gamma_{1,1}\) (of discriminant 1), and hence we have a splitting
\[
\Gamma_{1,9} \cong \Gamma_{1,1} \oplus E_8
\]
where \(E_8\) is the unique even, unimodular, negative-definite lattice corresponding to the Dynkin diagram \(E_8\). We should remark that the classes \(C_0\) and \(3H - \sum_{i=0}^8 C_i\) are the same as the base and fibre classes \(C'', E''\) of the rational elliptic surface discussed earlier.

\begin{rem}
The canonical divisor on the surface \(X_C\) is given by
\[
K_{X_C} = -3H + \sum_{i=0}^8 C_i = -E''.
\]
\end{rem}

This allows us now to compute the relationship between the groups \(H_2(X_C,\ZZ)\) and \(H_2(X,\ZZ)\), which we will need later.

\begin{lem}\label{lem_pushforward}
The map \(H_2(X_C,\ZZ) \to H_2(X,\ZZ)\) is given by
\[
H_2(X_C,\ZZ) \cong \Gamma_{1,1} \oplus E_8 \xrightarrow{proj} \Gamma_{1,1} \subseteq H_2(X,\ZZ)
\]
where \(\Gamma_{1,1}\) includes via the identification \(C := (\iota_2)_*C'', E := (\iota_2)_*(E'')\) described earier.
\end{lem}

\begin{proof}
We first claim that \((\iota_2)_* C_i = C + E\) for \(1 \leq i \leq 8\), and that \((\iota_2)_*H = 3(C + E)\). This can be seen simply by using the push-pull formula and noting that
\[
\iota_2^*L_1 = C_0 \qquad \iota_2^* L_2 = - E'' \qquad \iota_2^* L_3 = E''.
\]
Now, an element \(aH + \sum_{i=0}^8 b_iC_i\) is in the orthogonal complement of the lattice spanned by \(E'' = 3H - \sum_{i=0}^8 C_i, C'' = C_0\) if and only if
\begin{enumerate}
\item \(b_0 = 0\)
\item \(3a + \sum_{i=0}^8 b_i = 0\).
\end{enumerate}
Thus, we have that
\begin{align*}
(\iota_2)_*(aH + \sum_{i=0}^8 b_iC_i) &= 3a(C+E) + \sum_{i=0}^8 b_i(C+E) \\
&= \Big(3a + \sum_{i=0}^8 b_i\Big)(C+E) = 0
\end{align*}
as claimed.
\end{proof}

Finally, we need one fact about effectivity of classes in \(H_2(X_C,\ZZ)\).

\begin{lem}\label{lem_effective_classes}
Let \(\beta = C'' + nE'' + \lambda \in H_2(X_C,\ZZ)\), where \(\lambda \in E_8\). Then \(\beta\) is effective if and only if \(\lambda \cdot \lambda \geq -2n\).
\end{lem}

\begin{proof}
This is a straightforward application of Riemann-Roch. For a divisor \(D\) on \(X_C\), this reads as
\[
\chi(D) = 1 + \frac{1}{2}D \cdot (D - K_{X_C}).
\]
In particular, for \(D = C'' + nE'' + \lambda\), we find that
\[
h^0(D) + h^2(K_{X_C} - D) = 1 + n + \frac{1}{2}\lambda \cdot \lambda.
\]
Thus since \(K_{X_C} - D\) will never be effective, it follows that as long as \(n + \frac{1}{2}\lambda \cdot \lambda \geq 0\), that \(D\) will have a section, and hence be effective.
\end{proof}

\subsection{The K3 fibration}

To compute the Gopakumar-Vafa invariants of \(X\) in the fibre-wise classes (i.e. those which project down to 0 under the map \(\pi : X \to C\)), we use the machinery of \cite{kmps_nlyz,mp_gwnl}, which we will review here. Moreover, the ideas in this section closely follow the the ideas of \cite{kmps_nlyz}. For more detail, that article is strongly recommended. 

\begin{defn}
Let \(\Lambda\) be a rank \(r\) lattice. A family of \(\Lambda\)-polarized K3 surfaces over a base curve \(\Sigma\) is a scheme \(Z\) over \(\Sigma\) together with a collection of line bundles \(L_1, \ldots, L_r\) such that, for each \(b \in \Sigma\), the fibre \((X_b, L_1|_{X_b}, \ldots, L_r|_{X_b})\) is a \(\Lambda\)-polarized K3 surface.
\end{defn}

Such a family \(Z \xrightarrow{\pi} \Sigma\) yields a map \(\iota_\pi\) to the moduli space \(\mathcal{M}_\Lambda\) of \(\Lambda\)-polarized K3 surfaces. Intersecting the image of the curve with certain divisors in \(\mathcal{M}_\Lambda\) (see again, \cite{kmps_nlyz,mp_gwnl}) will produce the {\em Noether-Lefschetz numbers}. These are given as follows.

The Noether-Lefschetz divisors consist of those \(\Lambda\)-polarized K3 surfaces which jump in Picard rank; these are determined by 
\begin{enumerate}
\item an integer \(h\), such that the square of the new class \(\beta\) is \(2h - 2\)
\item \(r\) integers \(d_1, \ldots, d_r\) which are given by \(d_i = \int_\beta L_i\).
\end{enumerate}
We denote such a divisor by \(D_{h;d_1, \ldots, d_r}\), and we then define
\[
NL_{(h;d_1, \ldots, d_r)}^\pi = \int_{\iota_\pi \Sigma} D_{h;d_1, \ldots, d_r}.
\]
It should be remarked that, by the Hodge index theorem, this will be only be non-zero if the discriminant
\[
\Delta(h;d_1, \ldots, d_r) = (-1)^r 
\det \left(\:
\begin{array}{*{13}{c}}
\cline{1-3}
\multicolumn{1}{|c}{ } & & \multicolumn{1}{c|}{ }& d_1 \\
\multicolumn{1}{|c}{ } & \Lambda & \multicolumn{1}{c|}{ }& \vdots \\
\multicolumn{1}{|c}{ } & & \multicolumn{1}{c|}{ }& d_r \\
\cline{1-3}
d_1 & \cdots & d_r & 2h - 2
\end{array}
\right)
\]
is non-negative.

Let \(r_{0,h}\) denote the reduced Gopakumar-Vafa invariants of a K3 surface in a class \(\beta\) such that \(\beta \cdot \beta = 2h-2\). From \cite{kmps_nlyz}, these only depend on the square of \(\beta\) (and not its divisibility, as one might expect) , and they satisfy the {\em Yau-Zaslow formula} (see \cite{bryan_leung_k3, gottsche, kmps_nlyz, yau_zaslow})
\[
\sum_{h=0}^\infty r_{0,h} q^{h-1} = \frac{1}{\Delta(q)} = \frac{1}{\eta(q)^{24}}=q^{-1} + 24 + 324q + 3200q^2 + \cdots.
\]
It should be remarked that the power of 24 that shows up in this formula is due to the presence of the 24 nodal fibres in our elliptically fibred K3 surfaces.

Let \(n_{(d_1, \ldots, d_r)}^Z\) denote the Gopakumar-Vafa invariants of the threefold \(Z\) defined by
\[
n_{(d_1, \ldots, d_r)}^Z =\sum_{
\substack{\beta \in H_2(Z) \\ \int_\beta L_i = d_i}} n_\beta^Z.
\]
We have the following relation between these invariants.

\begin{thm}{\cite[Theorem \(1^*\)]{mp_gwnl}}\label{thm_gwnl}
The invariants \(n_{(d_1, \ldots, d_r)}^Z, r_{0,h}\), and  \(NL_{h;d_1, \ldots, d_r}^\pi\) satisfy the following relationship.
\[
n_{(d_1, \ldots, d_r)}^Z = \sum_{h=0}^\infty r_{0,h} NL_{h;d_1, \ldots, d_r}^\pi
\]
\end{thm}

Consider now the restriction of \(L_1, L_2\) to \(X_F\). We an compute their intersections via Lemma \ref{lem_triple_intersections} to find that we have
\[
\begin{pmatrix} 
L_1 \cdot L_1 & L_1 \cdot L_2 \\
L_2 \cdot L_1 & L_2 \cdot L_2
\end{pmatrix}
=
\underbrace{
\begin{pmatrix} 
-2 & 1 \\
1 & 0
\end{pmatrix}
}_\Lambda
\]
where in this case the rank \(r\) of \(\Lambda\) is 2.

What we would like to have is that the triple \((X, L_1, L_2)\) is a family of \(\Lambda\)-polarized K3 surfaces. However, due to the presence of singular fibres (due to the singularities of \(\Delta = 4g_2^3 - 27g_3^2\)) this is not the case. However, we can ``resolve'' this threefold (see \cite{kmps_nlyz, mp_gwnl}) to obtain a threefold \(\widetilde{X} \xrightarrow{\widetilde{\pi}} C\) which is such a family. We can then relate the invariants of the two families as follows, allowing us to compute the Gopakumar-Vafa invariants of \(X\) as desired.

\begin{lem}\label{lem_gv_resolution}
The invariants of \(X, \widetilde{X}\) satisfy
\[
n_{(d_1, d_2)}^{\widetilde{X}} = 2n_{(d_1, d_2)}^X.
\]
\end{lem}

Our final ingredient is to note that the Noether-Lefschetz numbers are coefficients of a modular form of weight \(\frac{22 - r}{2} = 10\); that is, they are the coefficients of some multiple \(E_4(z)E_6(z) = E_{10}(z) = 1 - 264q -135432q^2 - \cdots\). Thus we need to only compute a single such coefficient to determine all of the Noether-Lefschetz numbers.

\begin{defn}
Let \(f(z) = \sum_{n=0}^\infty a_nz^n\). Then we will use the notation
\[
[n]f(z) = a_n
\]
to denote the coefficient of \(z^n\) in \(f(z)\).
\end{defn}

Using this notation, we have the following lemma.

\begin{lem}\label{lem_nl000}
We have that
\[
NL_{0;0,0}^{\widetilde{X}} = 1056
\]
and so consequently we have that
\[
NL_{h;d_1,d_2}^{\widetilde{\pi}} = -4\bigg[\frac{\Delta(h;d_1,d_2)}{2}\bigg]E_{10}(z).
\]
\end{lem}

\begin{proof}
The proof of this is identical to the proofs of Lemma 2 and Proposition 2 of \cite{kmps_nlyz}, and thus we omit it.
\end{proof}

\section{Computations of the generating functions}

We are now ready to compute the generating functions for the Gopakumar-Vafa invariants. The generating functions we are interested are those of the following form.

Choose \(\beta = mC' + rF' \in H_2(S)\) (which we will identify from now on for simplicity's sake with its image in \(H_2(X,\ZZ)\)). Define
\[
F_\beta(q) = \sum_{n =0}^\infty n_{\beta + nE}^Xq^{n - m - \frac{1}{2}r}.
\]

\begin{rem}
We choose this shift in the exponent of \(q\) to match the results in \cite{KMW}. This ensures that the generating functions that we obtain below will be modular, but we don't have a better interpretation of this shifted power.
\end{rem}

We then have the following theorem.

\begin{thm}
We have the following expressions for generating functions \(F_\beta(q)\):
\begin{gather*}
F_F(q) = -2\frac{E_{10}(q)}{\Delta(q)} = -2q^{-1} + 480 + 282888q + 17058560q^2 + ,\cdots\\
F_C(q) = \frac{E_4(q)}{\sqrt{\Delta(q)}} = q^{-\frac{1}{2}} + 252q^{\frac{1}{2}} + 5130q^{\frac{3}{2}} + 54760q^{\frac{5}{2}} + .\cdots
\end{gather*}
Each of these is a meromorphic modular form of weight -2, and moreover each of the generating functions \(F_{mF}(q^m)\) is also (meromorphic) modular of the same weight, but for the group \(\Gamma_1(m^2)\).
\end{thm}

The first two of these generating functions are conjectured (with physical justification) in the papers \cite{KMW,KMV}, along with a few others. We have not found any prior description of the third, although it is an easy generalization of the first.

We will split the proof of this theorem up into several parts.

\begin{thm}\label{thm_pt1}
We have the equality
\[
F_F(q) = -2\frac{E_{10}(q)}{\Delta(q)}.
\]
\end{thm}

\begin{proof}
We will compute first the function \(F_F(q)\). Since we have that the class \(F + nE\) is determined uniquely by its integration against \(L_1, L_2\), we have that
\[
n_{F + nE}^X = n_{(n-2,1)}^X.
\]
Combining Lemmata \ref{lem_nl000}, \ref{lem_gv_resolution}, and Theorem \ref{thm_gwnl}, our generating function \(F_F(q)\) is given by
\begin{align*}
F_F(q) &= \sum_{n=0}^\infty n_{(n-2,1)}^Xq^{n-1} \\
&= \sum_{n=0}^\infty \tfrac{1}{2}n_{(n-2,1)}^{\widetilde{X}}q^{n-1} \\
&= \frac{1}{2}\sum_{n=0}^\infty \sum_{h=0}^\infty r_{0,h}NL_{h;n-2,1}^{\widetilde{\pi}} q^{n-1}.
\end{align*}
We can compute the discriminant \(\Delta(h;n-2,1) = 2n - 2h\) which must be non-negative, so the summation is really over those \(n, h \geq 0\) with \(n \geq h\). Thus we can write this as
\begin{align*}
F_F(q) &= \frac{1}{2}\sum_{h=0}^\infty r_{0,h} \sum_{n=h}^\infty NL_{h;n-2,1}^{\widetilde{\pi}} q^{n-1} \\
&= \frac{1}{2} \sum_{h=0}^\infty r_{0,h}q^{h-1} \sum_{n=h}^\infty (-4)\bigg[\frac{2n-2h}{2}\bigg] E_{10}(z)q^{n-h} \\
&= -2 \sum_{h=0}^\infty r_{0,h}q^{h-1}\sum_{n=h}^\infty [n-h]E_{10}(z)q^{n-h} \\
&= -2\sum_{h=0}^\infty r_{0,h}q^{h-1} E_{10}(q) \\
&= -2\frac{E_{10}(q)}{\Delta(q)}.
\end{align*}
\end{proof}

To prove the next formula, we need the following computation of the Gromov-Witten invariants (for primitive classes) of a rational elliptic surface.

\begin{thm}[\cite{bryan_leung_k3}, Theorem 6.2]
The generating function for the Gopakumar-Vafa invariants of the rational elliptic surface \(X_C\) in the classes \(C'' + nE''\) is given by
\[
\sum_{n=0}^\infty n_{C''+nE''}^{X_C}q^{n-\frac{1}{2}} = \frac{1}{\sqrt{\Delta(q)}}.
\]
\end{thm}

We now prove the following.

\begin{thm}\label{thm_pt2}
We have the equality
\[
F_C(q) = \frac{E_4(q)}{\sqrt{\Delta(q)}}.
\]
\end{thm}

\begin{proof}
Recall from Lemma \ref{lem_pushforward} that the map \((\iota_2)_* : H_2(X_C,\ZZ) \cong \Gamma_{1,1} \oplus E_8 \to H_2(X,\ZZ)\) is essentially the projection onto the \(\Gamma_{1,1}\) factor.

Now, we obtain curves in \(X\) in the class \(C + nE\) by considering curves in \(X_C\) in some effective class which pushes forward to this class; from Lemma \ref{lem_pushforward}, these will be curves of the form \(\beta = C'' + nE'' + \lambda\) where \(\lambda \in E_8\). From Lemma \ref{lem_effective_classes}, we know that the effective ones are those with \(n \geq -\frac{1}{2}\lambda \cdot \lambda\).

We can now compute that
\begin{align*}
F_C(q) &= \sum_{n=0}^\infty n_{C + nE}^X q^{n-\frac{1}{2}} \\
&= \sum_{n=0} \sum_{\substack{\beta \in H_2(X_C) \\ (\iota_2)_*\beta = C + nE}} n_\beta^{X_C}q^{n-\frac{1}{2}} \\
&= \sum_{n=0}^\infty \sum_{\substack{\lambda \in E_8 \\ -\lambda \cdot \lambda \leq 2n}} n_{C'' + nE'' + \lambda}^{X_C} q^{n-\frac{1}{2}}.
\end{align*}

Now, morally similar to the case of K3 surfaces, from primitive curve classes the invariants \(n_\beta^{X_C}\) only depend on the square of \(\beta\). In particular, any such curve can be transformed into one of the form \(\beta = C'' + nE''\) by a series of Cremona transformations and permutations of the exceptional classes (see \cite{gp}). It follows then that
\[
n_{C'' + nE'' + \lambda}^{X_C} = n_{C'' + (n + \tfrac{1}{2}\lambda\cdot\lambda)E''}^{X_C}
\]
and so the generating function becomes
\begin{align*}
F_C(q) &= \sum_{n=0}^\infty \sum_{\substack{\lambda \in E_8\\-\lambda\cdot\lambda\leq 2n}} n_{C_0 + (n + \frac{1}{2}\lambda \cdot \lambda)E}^{X_C}q^{n-\frac{1}{2}} \\
&= \sum_{\lambda \in E_8} q^{-\frac{1}{2}\lambda \cdot \lambda} \sum_{n=-\frac{1}{2}\lambda \cdot \lambda}^\infty n_{C_0 + (n + \tfrac{1}{2}\lambda \cdot \lambda)E}^{X_C} q^{n + \frac{1}{2}\lambda \cdot \lambda - \frac{1}{2}} \\
&= \sum_{\lambda \in E_8} q^{-\frac{1}{2}\lambda \cdot \lambda} \sum_{n=0}^\infty n_{C_0 + nE}^{X_C} q^{n-\frac{1}{2}} \\
&= \Theta_{E_8}(q)\Big(\frac{1}{\sqrt{\Delta(q)}}\Big) = \frac{E_4(q)}{\sqrt{\Delta(q)}}
\end{align*}
as claimed (with the last equality being due to the well-known fact that \(\Theta_{E_8}(q) = E_4(q)\)).
\end{proof}

To prove the last statement in the theorem, we need a little extra notation.

\begin{defn}
Let \(f(z) = \sum_{a_nz^n}\) be a power series, and let \(m, k\) be integers with \(0 \leq k < m\). We define then
\[
f_{m,k}(z) = \sum_{\substack{n \equiv k\\\pmod m}} a_nz^n = \sum_{n=0}^\infty a_{mn+k}z^{mn+k}.
\]
\end{defn}

We should note that in the case that \(f(z)\) is a modular form of weight \(r\) for \(SL_2(\ZZ)\), then each of the functions \(f_{m,k}(z)\) are also modular of the same weight for the subgroup \(\Gamma_1(m^2)\).

Furthermore, we can expand this definition for values of \(k\) outside of the given range by replacing \(k\) with a suitable integer congruent to \(k \pmod m\) within that range. For example, \(f_{m,-1}(z) = f_{m,m-1}(z)\).

We can now state more precisely our final theorem.

\begin{thm}\label{thm_pt3}
Let \(m > 1\). Then the generating function \(F_{mF}(q)\) is given by
\[
F_{mF}(q) = \sum_{n=0}^\infty n_{mF + nE}q^{n-m} \\
=
-2\sum_{\ell=0}^{m-1}\bigg(\frac{1}{\Delta(u)}\bigg)_{m, \ell-1}\big(E_{10}(u)\big)_{m, 1 - \ell}
\]
where \(q = u^m\).
\end{thm}


\begin{proof}
This proof follows very similarly to the proof of Theorem \ref{thm_pt1}. We similarly begin with noting that \(n_{mF + nE}^X = n_{(n-2m,m)}^X\), which allows us to write
\begin{align*}
F_{mF}(q) &= \sum_{n=0}^\infty n_{(n-2m,m)}^Xq^{n-m} \\
&= \sum_{n=0}^\infty \tfrac{1}{2}n_{(n-2m,m)}^{\widetilde{X}}q^{n-m} \\
&= \frac{1}{2}\sum_{n=0}^\infty \sum_{h=0}^\infty r_{0,h}NL_{h;n-2m,m}^{\widetilde{\pi}}q^{n-m} .
\end{align*}
In this case, the discriminant \(\Delta(h;n-2m,m) = 2 - 2h + 2nm - 2m^2\) which as usual must be non-negative, leaving us summing over all pairs \((h,n)\) such that \(n \geq m + \frac{h-1}{m}\). Thus we obtain
\[
F_{mF}(q) = -2 \sum_{h=0}^\infty r_{0,h} \sum_{n\geq m + \frac{h-1}{m}} [1 - h + nm - m^2]E_{10}(z)q^{n-m}.
\]
To simplify this further, we split the summation over \(h\) into a sum over congruence classes mod \(m\). If we let \(q = u^m\), then this yields the following.
\begin{align*}
F_{mF}(q) &= -2 \sum_{\ell = 0}^{m-1}\sum_{h=0}^\infty r_{0,mh+\ell} \sum_{n\geq m + h+\frac{\ell-1}{m}} [1 - \ell + m(n- h - m)]E_{10}(z)u^{nm-m^2} \\
&= -2 \sum_{\ell = 0}^{m-1}\sum_{h=0}^\infty r_{0,mh+\ell} u^{mh + \ell - 1}\sum_{n\geq m + h+\frac{\ell-1}{m}} [1 - \ell + m(n-h-m)]E_{10}(z)u^{m(n - h -m)- \ell+1} \\
&= -2\sum_{\ell=0}^{m-1}\sum_{h=0}^\infty r_{0,mh+\ell}u^{mh + \ell - 1} \big(E_{10})_{m, 1 - \ell}(u) \\
&= -2\sum_{\ell=0}^{m-1}\bigg(\frac{1}{\Delta(u)}\bigg)_{m, \ell-1}\big(E_{10}(u)\big)_{m, 1 - \ell}
\end{align*}
which ends the proof.
\end{proof}

The above results show that we end up with meromorphic modular forms when we consider generating functions for Gopakumar-Vafa invariants for curve classes of the form \(mF + nE\) and \(C + nE\). From the conjectured results in \cite{KMW}, it seems that we should end up with similar results for curve classes of the form \(rC + nE\); a natural approach to study these would be to use the recursion of \cite{gp}, which we will look to do at a future date.

%

%
%
%

\bibliographystyle{amsplain}
\bibliography{delp13}

\end{document}